\documentclass[12pt]{article}

\usepackage{amsmath}
\usepackage{amssymb}
\usepackage{amsthm}
\usepackage{tikz}
\usetikzlibrary{matrix}

\newtheorem{lemma}{Lemma}
\newtheorem{theorem}{Theorem}
\newtheorem{remark}{Remark}

\def\Q{\mathbb{Q}}

\def\eps{\varepsilon}
\DeclareMathOperator{\Cl}{Cl}
\def\errorexp{1-\frac{2}{n+1}}

\title{An improved error term for counting $D_4$-quartic fields}
\author{Kevin J.~McGown \and Amanda Tucker}

\begin{document}

\maketitle

\begin{abstract}
We prove that the number of quartic fields $K$ with
discriminant $|\Delta_K|\leq X$ whose Galois closure is $D_4$  equals $CX+O(X^{5/8+\eps})$,
improving the error term in a well-known result of Cohen, Diaz y Diaz, and Olivier.
We prove an analogous result for counting quartic dihedral extensions over an arbitrary base field.
\end{abstract}

\section{Introduction}\label{S:intro}

Let $N_n(G,X)$ be the number of degree $n$ number fields with Galois closure $G$ and discriminant
$|\Delta_K|\leq X$.  It is an interesting problem to find asymptotic expressions for $N_n(G,X)$.
This is the subject of conjectures of Malle and Bhargava.
Among other things, this is connected to the inverse Galois problem and the Cohen--Lenstra heuristics for class groups.
See, for example, the papers~\cite{bhargava.mass, cohen.lenstra, malle1, malle2, kluners}.

In this investigation, we will focus on quartic fields whose Galois closure is $D_4$, the symmetry group of a square. 
Our first result is as follows:

\begin{theorem}\label{T:1}
We have
$$
N_4(D_4,X)=CX+O(X^{5/8+\eps})
$$
where
$$
C=
\frac{1}{2}
\sum_{\substack{[k:\Q]=2}}
\frac{1}{2^{r_2(k)}\Delta_k^2}
\frac{\zeta_k^*(1)}{\zeta_k(2)}
\,.
$$
\end{theorem}
Here $\zeta_k(s)$ denotes the Dedekind zeta function of $k$ and $\zeta^*(1)$ denotes the
first non-vanishing Laurent coefficient of $\zeta(s)$ at $s=1$.  As is customary,
$r_1(k)$ and $2r_2(k)$ denote the number of real and complex embeddings, respectively. 

Previously, Cohen, Diaz y Diaz, and Olivier (see~\cite{cohen.diaz.olivier}) proved that
$N_4(D_4,X)=CX+O(X^{3/4+\eps})$.
In later numerical work (see~\cite{cohen.diaz.olivier2}), they suggest that it is ``reasonable to conjecture'' 
that the error is $O(X^{1/2+\eps})$, and that there may be a secondary term.
Theorem~\ref{T:1} strengthens the error term in their result.
Our proof follows their approach and appeals to some of the calculations in Section 3 of~\cite{cohen.diaz.olivier},
but we do not  make explicit use of the same Dirichlet series.
See Remark~\ref{R:1} in~\S\ref{S:proof} for additional comments on their conjecture in the context of our proof.

In the course of proving Theorem~\ref{T:1}, we establish a result for counting relative
quadratic extensions for which the implicit constant in the error term only depends
on the degree of the base field.  The original asymptotic for this
counting problem was given in~\cite{wright}.

\begin{theorem}\label{T:relative}
Let $k$ be a number field of degree $n\geq 2$.  We have
\begin{align*}
&
\sum_{\substack{[K:k]=2\\N(\Delta_{K/k})\leq X}}1=
\frac{1}{2^{r_2(k)}}
\frac{\zeta_k^*(1)}{\zeta_k(2)}X
\;+
\\
&
|\Cl(k)[2]|
\cdot
\begin{cases}
O\left(
|\Delta_k|^{1/3}\log|\Delta_k|\, X^{1/2}\log X
\right)
&
n=2\,,
\\
O\left(
|\Delta_k|^{1/4}(\log|\Delta_k|)^2\, X^{1/2}(\log X)^3
\right)
&
n=3\,,
\\
O_n\left(
|\Delta_k|^{\frac{1}{n+1}}X^{\errorexp}(\log X)^{n-1}
\right)
&
n>3\,.
\end{cases}
\end{align*}
\end{theorem}
In the above, $\Cl(k)[2]$ denotes the $2$-torsion subgroup of the class group of $k$.

The referee pointed out that our methods also prove a similar result for
$D_4$-quartic extensions over an arbitrary base field, and therefore
we also include the following:

\begin{theorem}\label{T:3}
Let $F$ be a number field of degree $n\geq 2$.
The number of quartic extensions $K/F$ with $N(\Delta_{K/F})\leq X$
whose Galois closure is $D_4$ equals
$$
\left(\frac{1}{2}\sum_{[k:F]=2}
\frac{1}{2^{r_2(k)}N(\Delta_{k/F})^2}\frac{\zeta^*_k(1)}{\zeta_k(2)}
\right)
X
+
O_n\left(
|\Cl(F)[2]|^3\cdot|\Delta_F|^{\frac{2}{2n+1}+\eps}X^{1-\frac{2}{2n+1}+\eps}
\right)
\,.
$$
\end{theorem}
In the above, $\Delta_{L/M}$ denotes the relative discriminant and $N$
denotes the absolute norm.

We note in passing that other variations on the problem
of counting $D_4$-quartic fields have been investigated,
including counting the fields when ordered by their Artin conductor
(see~\cite{altug.ali.shankar.varma.wilson, BFSLV22, Friedrichsen}).

Late in the preparation of this manuscript we became aware that
Barquero-Sanchez, Masri, and Thorne
(see Theorem 2.4 of~\cite{barquero.masri.thorne})
already obtained an improvement to $O(X^{5/7 + \eps})$
for the closely related problem of counting $D_4$-quartic fields that
are totally imaginary extensions of totally real quadratic fields.

\section{Initial setup}\label{S:initial}
The following is a truncated version of Corollary~2.2 of~\cite{cohen.diaz.olivier}:
\begin{equation}\label{E:fields}
\sum_{\substack{[k:\Q]=2\\
|\Delta_{k}|\leq \sqrt{X}}}
\;\;
\sum_{\substack{[K:k]=2\\N(\Delta_{K/k})\leq X/\Delta_k^2}}1
=
2\sum_{\substack{[K:\Q]=4\\G(K/\Q)\simeq D_4\\|\Delta_{K}|\leq X}} 1
+
\sum_{\substack{[K:\Q]=4\\G(K/\Q)\simeq C_4\\|\Delta_{K}|\leq X}} 1
+
3\sum_{\substack{[K:\Q]=4\\G(K/\Q)\simeq V_4\\ |\Delta_{K}|\leq X}} 1
\,.
\end{equation}
In words, counting extensions that are quadratic over quadratic picks
up every $D_4$-quartic (up to isomorphism) exactly twice, every $V_4$-quartic 
exactly three times, and every $C_4$-quartic exactly once.  This is observed in
the following field diagram for the Galois closure $L$ of a generic $D_4$-quartic field $K_1$.
(Every line stands for a degree $2$ extension.)

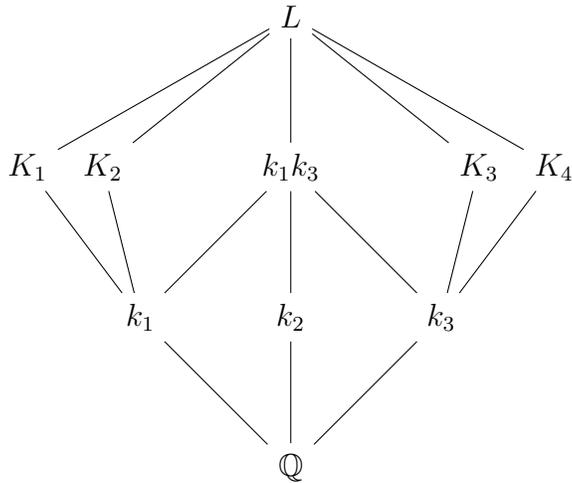
\begin{center}
\begin{tikzpicture}
    \node (Q1) at (0,0) {$\Q$};
    \node (Q2) at (-2,2) {$k_1$};
    \node (Q3) at (0,2) {$k_2$};
    \node (Q4) at (2,2) {$k_3$};
    
    \node (Q5) at (-3.5,4) {$K_1$};
    \node (Q6) at (-2.5,4) {$K_2$};
    \node (Q7) at (0,4) {$k_1k_3$};
    \node (Q8) at (2.5,4) {$K_3$};
    \node (Q9) at (3.5,4) {$K_4$};
    \node (Q10) at (0,6) {$L$};

    \draw (Q1)--(Q2);
    \draw (Q1)--(Q3);
    \draw (Q1)--(Q4);
        
    \draw (Q2)--(Q5);
    \draw (Q2)--(Q6);

    \draw (Q2)--(Q7);
    \draw (Q3)--(Q7);
    \draw (Q4)--(Q7);

    \draw (Q4)--(Q8);
    \draw (Q4)--(Q9) ;

    \draw (Q5)--(Q10);
    \draw (Q6)--(Q10);

    \draw (Q8)--(Q10);
    \draw (Q9)--(Q10) ;

    \draw (Q7)--(Q10);

\end{tikzpicture}
\end{center}
The truncations in the subscripts in~(\ref{E:fields})
follow immediately from
the identity $\Delta_K=N(\Delta_{K/k})\Delta_k^2$.

The first term on the right of~(\ref{E:fields}) is what we want to count.
The next two terms on the right are $O(\sqrt{X})$ and $O(\sqrt{X}(\log X)^2)$,
respectively (see~\cite{baily, Maki}).
The main task is to deal with the innermost sum on the left, which counts relative
quadratic extensions.

\section{Counting relative quadratic extensions}

In this section, we consider the problem of counting quadratic extensions over an arbitrary base field $k$
of degree $n$, which will ultimately lead to the proof of Theorem~\ref{T:relative}.

\subsection{Parametrizing relative quadratics}\label{S:param}

We parametrize all quadratic extensions $K/k$ in terms of data from $k$.
(We refer the reader to~\cite{cohen.diaz.olivier} for the proofs of the results in this subsection.)
Let $V(k)$ denote the set of all $u\in k^*$ such that $(u)=\mathfrak{q}^2$
for some ideal $\mathfrak{q}$.
Write $S(k)=V(k)/(k^*)^2$.
This is the $2$-Selmer group of $k$.
Given an element $\overline{u}\in S(k)$ we will always tacitly assume $(u,2)=1$.
Let $A(k)$ denote the set of all integral squarefree ideals $\mathfrak{a}$
such that $\overline{\mathfrak{a}}\in\Cl(k)^2$.

There is a bijection 
$$
  A(k)\times S(k) \to \left\{[K:k]\leq 2\right\}
  \,.
$$
Under this map, a pair $(\mathfrak{a},\overline{u})$ corresponds to an extension $K/k$,
and 
the ``identity'' corresponds to the trivial extension.

We wish to determine the discriminant $\Delta_{K/k}$ in terms of $(\mathfrak{a},\overline{u})$.
To this end, we define an ideal $\mathfrak{c}(\mathfrak{a},\overline{u})$
in the following way.
Given $(\mathfrak{a},\overline{u})$, first write $\mathfrak{a}\mathfrak{q}^2=(\alpha_0)$ with
$(\mathfrak{q},2)=1$; then let $\mathfrak{c}$ be the largest ideal such that
$\mathfrak{c}\mid 2$, $(\mathfrak{c},\mathfrak{a})=1$, and
$x^2\equiv \alpha_0 u\pmod{\mathfrak{c}^2}$ is solvable (in the multiplicative sense).
One checks that this definition is independent of the choices involved.  Then one has
$$
  \Delta_{K/k}=\frac{4\mathfrak{a}}{\mathfrak{c}^2}
  \,.
$$

\subsection{A formula for the number of relative quadratics}

Using the parametrization from~\S\ref{S:param},
we derive an expression for the number of quadratic extensions 
$K/k$ with $N(\Delta_{K/k})\leq X$.
The reader should compare this result to Theorem~1.1 of~\cite{cohen.diaz.olivier}.

\begin{lemma}\label{L:first}
Let $k$ be a number field of degree $n$.  One has
\begin{equation}\label{E:4}
\sum_{\substack{[K:k]=2\\N(\Delta_{K/k})\leq X}}
\hspace{-2ex}
1
=
-1+
2^{r_1(k)+r_2(k)}
\sum_{\substack{\mathfrak{d}\mid 2}}
\frac{1}{N(\mathfrak{d})}
\hspace{-1ex}
\sum_{\substack{\chi\in\widehat{\Cl_{\mathfrak{d}^2}(k)}\\\chi^2=\chi_0 }}
\sum_{\substack{\mathfrak{c}\mid \mathfrak{d}}}
\mu\left(\frac{\mathfrak{d}}{\mathfrak{c}}\right)
\hspace{-2ex}
\sum_{\substack{\mathfrak{a}\text{ squarefree}\\
N(\mathfrak{a})\leq \frac{N(\mathfrak{c}^2)}{4^n}X}}
\hspace{-2ex}
\chi(\mathfrak{a})
\,,
\end{equation}
where
$\Cl_{\mathfrak{d}^2}(k)$ is the ray class group of $k$ modulo $\mathfrak{d}^2$, and 
$\sum_\chi$ is over the characters of $\Cl_{\mathfrak{d}^2}(k)$ satisfying $\chi^2=\chi_0$.
\end{lemma}

\begin{proof}
For notational ease, we write $A=A(k)$ and $S=S(k)$.
We observe
\begin{align}
\label{E:0}
\sum_{\substack{[K:k]=2\\N(\Delta_{K/k})\leq X}}1
&=
-1+
\sum_{\mathfrak{a}\in A}\sum_{\substack{\overline{u}\in S\\
N\left(\frac{4\mathfrak{a}}{\mathfrak{c}(\mathfrak{a},\overline{u})^2}\right)\leq X}} 1
\\
\label{E:1}
&=
-1+
\sum_{\mathfrak{c}\mid 2}
\sum_{\substack{\mathfrak{a}\in A\\  N(\mathfrak{a})\leq \frac{N(\mathfrak{c}^2)}{4^n}X}}
\sum_{\substack{\overline{u}\in S\\\mathfrak{c}(\mathfrak{a},\overline{u})=\mathfrak{c}}} 1
\,.
\end{align}
First we deal with the innermost sum of (\ref{E:1}),
which is nonzero only if \mbox{$(\mathfrak{a},\mathfrak{c})=1$}.
For fixed $\mathfrak{c}$ and $\mathfrak{a}$
with $\mathfrak{c}\mid 2$ and $(\mathfrak{c},\mathfrak{a})=1$ we have
$$
\sum_{\substack{\overline{u}\in S\\ \exists x\;x^2\equiv\alpha_0u\pmod{\mathfrak{c}^2}}}
1
=
\sum_{\substack{\mathfrak{c}\mid\mathfrak{d}\mid 2\\(\mathfrak{d},\mathfrak{a})=1}}
\sum_{\substack{\overline{u}\in S\\ \mathfrak{c}(\mathfrak{a},\overline{u})=\mathfrak{d}}}
1
=
\sum_{\substack{\mathfrak{d}\mid\frac{2}{\mathfrak{c}}\\(\mathfrak{d},\mathfrak{a})=1}}
\sum_{\substack{\overline{u}\in S\\ \mathfrak{c}(\mathfrak{a},\overline{u})=\mathfrak{d}\mathfrak{c}}}
1
$$
and therefore,
using M\"obius inversion, we obtain
\begin{equation}\label{E:2}
\sum_{\substack{\overline{u}\in S\\ \mathfrak{c}(\mathfrak{a},\overline{u})=\mathfrak{c}}}
1
=
\sum_{\substack{\mathfrak{d}\mid\frac{2}{\mathfrak{c}} \\(\mathfrak{d},\mathfrak{a})=1}}
\mu(\mathfrak{d})
\sum_{\substack{\overline{u}\in S\\ \exists x\;x^2\equiv\alpha_0u\pmod{\mathfrak{c}^2\mathfrak{d}^2}}}
1
\,.
\end{equation}

Plugging~(\ref{E:2}) into~(\ref{E:1}) yields
\begin{align}
\nonumber
\sum_{\substack{[K:k]=2\\N(\Delta_{K/k})\leq X}}1
&=
-1+
\sum_{\mathfrak{c}\mid 2}
\sum_{\substack{\mathfrak{a}\in A\\ (\mathfrak{a},\mathfrak{c})=1\\ N(\mathfrak{a})\leq \frac{N(\mathfrak{c}^2)}{4^n}X}}
\sum_{\substack{\mathfrak{d}\mid\frac{2}{\mathfrak{c}} \\(\mathfrak{d},\mathfrak{a})=1}}
\mu(\mathfrak{d})
\sum_{\substack{\overline{u}\in S\\ \exists x\;x^2\equiv\alpha_0u\pmod{\mathfrak{c}^2\mathfrak{d}^2}}}
1
\\
\label{E:2b}
&=
-1+
\sum_{\substack{\mathfrak{e}\mid 2}}
\sum_{\substack{\mathfrak{c}\mid\mathfrak{e}}}
\mu\left(\frac{\mathfrak{e}}{\mathfrak{c}}\right)
\sum_{\substack{\mathfrak{a}\in A\\(\mathfrak{e},\mathfrak{a})=1\\  N(\mathfrak{a})\leq \frac{N(\mathfrak{c}^2)}{4^n}X}}
\sum_{\substack{\overline{u}\in S\\ \exists x\;x^2\equiv\alpha_0u\pmod{\mathfrak{e}^2}}}
1
\,.
\end{align}

Proposition 3.9 and Lemma~3.10 of~\cite{cohen.diaz.olivier} together
with the orthogonality relations give
\begin{align}
\nonumber
\sum_{\substack{\overline{u}\in S\\ \exists x\;x^2\equiv\alpha_0u\pmod{\mathfrak{c}^2}}}
1
&=
\begin{cases}
2^{r_1(k)+r_2(k)}|\Cl_{\mathfrak{c}^2}(k)[2]|N(\mathfrak{c})^{-1}
&
\overline{\mathfrak{a}}\in\Cl_{\mathfrak{c}^2}(k)^2
\\
\label{E:3}
0
&
\overline{\mathfrak{a}}\not\in\Cl_{\mathfrak{c}^2}(k)^2
\end{cases}
\\
&=
\frac{2^{r_1(k)+r_2(k)}}{N(\mathfrak{c})}
\sum_{\substack{\chi\in\widehat{\Cl_{\mathfrak{c}^2}(k)}\\\chi^2=\chi_0 }}
\chi(\mathfrak{a})
\,.
\end{align}
Pugging (\ref{E:3}) into (\ref{E:2b}) leads to the desired result.
\end{proof}

\section{Sums of ray class characters}\label{S:a1}
Using the formula given in Lemma~\ref{L:first} as a starting point,
our goal is to derive the asymptotic expression for the number of relative quadratics
given in Theorem~\ref{T:relative}.
In order to do this, we will establish two lemmas on sums of ray class characters.

Let $\Phi$ denote the standard generalization of Euler's totient function to the ideals of $k$.
Let $\chi_0$ denote the principal character modulo
$\mathfrak{m}$.
It is well-known (see Satz XCV of~\cite{landau2}) that
\begin{equation}\label{E:Landau}
\sum_{\substack{
N(\mathfrak{a})\leq X }}
\chi_0(\mathfrak{a})
=\frac{\Phi(\mathfrak{m})}{N(\mathfrak{m})} \zeta_k^*(1)X + O_k(X^{\errorexp})
\,.
\end{equation}
However, the implicit constant depends on the number field $k$.  On the other hand,
recently a uniform version of Landau's method was given by Lowry-Duda, Taniguchi, and Thorne
(see~\cite{lowry-duda.taniguchi.thorne}).
Their Theorem~3 states
\begin{equation}\label{E:ltt3}
\sum_{\substack{N(\mathfrak{a})\leq X }}
1
=\zeta_k^*(1)X + O_n(|\Delta_k|^\frac{1}{n+1}X^{1-\frac{2}{n+1}}(\log X)^{n-1})
\,.
\end{equation}
Using (\ref{E:ltt3}) one can easily prove a uniform version of (\ref{E:Landau}).
We give such a result that holds for both principal and nonprincipal characters,
where the implicit constant only depends on $n=[k:\Q]$.

\begin{lemma}\label{L:second}
Let $k$ be a number field of degree $n\geq 2$.  Let $\chi\in\widehat{\Cl_\mathfrak{m}(k)}$ be
a ray class character of $k$ modulo $\mathfrak{m}$ with conductor $\mathfrak{f}$.
Let $\tau(\mathfrak{m})$ denote the number of ideal divisors of $\mathfrak{m}$.
We have
$$
  \sum_{N(\mathfrak{a})\leq X}\chi(\mathfrak{a})
  =
  \delta(\chi)
  \frac{\Phi(\mathfrak{m})}{N(\mathfrak{m})}\zeta_k^*(1)X+
  O_n(\tau(\mathfrak{m})(N(\mathfrak{f})|\Delta_k|)^{\frac{1}{n+1}}X^{\errorexp}(\log X)^{n-1})
  \,,
$$
where $\delta(\chi)$ equals $1$ if $\chi=\chi_0$ and $0$ otherwise.
\end{lemma}

\begin{proof}
Equation (\ref{E:ltt3}),
together with the identity
$$
  \sum_{\substack{
  N(\mathfrak{a})\leq X}}\chi_0(\mathfrak{a})
  =
  \sum_{\mathfrak{d}\mid\mathfrak{m}}\mu(\mathfrak{d})
  \sum_{N(\mathfrak{a})\leq X/N(\mathfrak{d})}1
  \,,
$$
allows one to conclude that
\begin{align*}
\sum_{\substack{
N(\mathfrak{a})\leq X }}
\chi_0(\mathfrak{a})
&=
\sum_{\mathfrak{d}\mid \mathfrak{m}}
\mu(\mathfrak{d})\zeta_k^*(1)\frac{X}{N(\mathfrak{d})}
+O_n\left(
\sum_{\mathfrak{d}\mid \mathfrak{m}}
|\Delta_k|^{\frac{1}{n+1}}
\left(\frac{X}{N(\mathfrak{d})}\right)^{1-\frac{2}{n+1}}(\log X)^{n-1}
\right)
\\
&=
\zeta_k^*(1)X\sum_{\mathfrak{d}\mid \mathfrak{m}}
\frac{\mu(\mathfrak{d})}{N(\mathfrak{d})}
\\
&\qquad
+O_n\left(
|\Delta_k|^{\frac{1}{n+1}}
X^{1-\frac{2}{n+1}}(\log X)^{n-1}
\sum_{\mathfrak{d}\mid \mathfrak{m}}
\left(\frac{1}{N(\mathfrak{d})}\right)^{1-\frac{2}{n+1}}
\right)
\,.
\end{align*}
The last factor in the error term is bounded by $\tau(\mathfrak{m})$,
and the result follows in the case where $\chi=\chi_0$.

We now turn to the case where the character is nonprincipal.
Invoking Theorem~5 of~\cite{lowry-duda.taniguchi.thorne}, we find that for a nonprincipal primitive ray class character $\psi$ with conductor $\mathfrak{f}$ one has
\begin{equation}\label{E:primitive}
\sum_{\substack{N(\mathfrak{a})\leq X }}
\psi(\mathfrak{a})
=
O_n((N(\mathfrak{f})|\Delta_k|)^{\frac{1}{n+1}}X^{\errorexp}(\log X)^{n-1})
\,.
\end{equation}
Indeed, Theorem~5 of~\cite{lowry-duda.taniguchi.thorne} applies to the setting of Hecke $L$-series,
as is indicated in the discussion following the statement of their Theorem~2.
One follows closely the proof of their Theorem 3, making the necessary modifications.
The most significant differences are the absence of a pole at $s=1$ and
the different functional equation.  Ultimately, this leads to replacing
$|\Delta_k|$ with $|\Delta_k| N(\mathfrak{f})$ in many of the equations.
One possible reference for the functional equation of a Hecke $L$-series is Chapter VII of~\cite{neukirch}.

To extend the estimate in $(\ref{E:primitive})$ to
all nonprincipal ray class characters $\chi$ modulo $\mathfrak{m}$,
one uses the identity
$$
  \sum_{N(\mathfrak{a})\leq X}\chi(\mathfrak{a})
  =
  \sum_{\substack{\mathfrak{d}\mid\mathfrak{m}}}\mu(\mathfrak{d})\psi(\mathfrak{d})
  \sum_{N(\mathfrak{a})\leq X/N(\mathfrak{d})}\psi(\mathfrak{a})
  \,,
$$
where $\psi$ is the primitive character inducing $\chi$.
The result follows.
\end{proof}

For our application, we will need to consider sums over squarefree ideals.
To that end, we establish the following lemma.
It will be convenient to define the $\mathfrak{m}$-imprimitive Dedekind zeta function
as $\zeta_k^\mathfrak{m}(s)=
\prod_{\mathfrak{p}\mid\mathfrak{m}}(1-N(\mathfrak{p})^{-s})
\zeta_k(s)
$.

\begin{lemma}\label{L:third}
Notation as in Lemma~\ref{L:second}.
We have
\begin{align*}
\sum_{\substack{\mathfrak{a} \text{ squarefree}\\N(\mathfrak{a})\leq X}}
\chi(\mathfrak{a})
&=
\delta(\chi)
\frac{\Phi(\mathfrak{m})}{N(\mathfrak{m})}\frac{\zeta_k^*(1)}{\zeta_k^\mathfrak{m}(2)}X
\;+
\\
&
\tau(\mathfrak{m})\cdot
\begin{cases}
O\left(
(N(\mathfrak{f})|\Delta_k|)^{1/3}\log|\Delta_k|
\cdot
X^{1/2}
\log X
\right)
&n=2\,,\\
O\left(
(N(\mathfrak{f})|\Delta_k|)^{1/4}(\log|\Delta_k|)^2
\cdot
X^{1/2}
(\log X)^3
\right)
&n=3\,,\\
O_n\left(
(N(\mathfrak{f})|\Delta_k|)^{\frac{1}{n+1}}
X^{1-\frac{2}{n+1}}
(\log X)^{n-1}
\right)
&n>3\,.\\
\end{cases}
\end{align*}
\end{lemma}

\begin{proof}
To deal with the squarefree condition, we employ the
following
\begin{align}
\label{E:sqrfree}
\sum_{\substack{\mathfrak{a} \text{ squarefree}\\N(\mathfrak{a})\leq X}}
\chi(\mathfrak{a})
&=
\sum_{\substack{N(\mathfrak{a})\leq X}}
\chi(\mathfrak{a})
\sum_{\mathfrak{d}^2\mid\mathfrak{a}}
\mu(\mathfrak{d})
\\
\nonumber
&=
\sum_{\substack{N(\mathfrak{d})\leq X^{1/2}}}
\mu(\mathfrak{d})
\sum_{\substack{N(\mathfrak{a})\leq X/N(\mathfrak{d})^2}}\chi(\mathfrak{d}^2)\chi(\mathfrak{a})
\,.
\end{align}
This identity holds for any character modulo $\mathfrak{m}$.
Hence, after applying Lemma~\ref{L:second},
the sum of interest on the lefthand side of~(\ref{E:sqrfree}) is equal to
\begin{align}
&
\label{E:another.first}
\delta(\chi)\sum_{\substack{N(\mathfrak{d})\leq X^{1/2}}}
\mu(\mathfrak{d})\chi(\mathfrak{d}^2)
\frac{\Phi(\mathfrak{m})}{N(\mathfrak{m})}\zeta_k^*(1)\frac{X}{N(\mathfrak{d})^2}
\\
&\label{E:another.second}
\qquad
+
O_n\left(
\sum_{N(\mathfrak{d})\leq X^{1/2}}
\tau(\mathfrak{m})
(N(\mathfrak{f})|\Delta_k|)^{\frac{1}{n+1}}
\left(\frac{X}{N(\mathfrak{d})^2}\right)^{1-\frac{2}{n+1}}(\log X)^{n-1}
\right)
\,.
\end{align}

It will be useful to write $g_\beta(Z):=\sum_{N(\mathfrak{d})\leq Z} N(\mathfrak{d})^\beta$
as we will need to consider this function for various values of $\beta$.
First, we have the estimate
\begin{equation}\label{E:BS}
g_0(Z)\ll_n Z(\log Z)^{n-1}
\,,
\end{equation}
which follows from Exercise~1 on page~231 of~\cite{borevich.shafarevich}.

In the case where $\chi=\chi_0$, the first summand (\ref{E:another.first}) above is
\begin{align*}
&
\frac{\Phi(\mathfrak{m})}{N(\mathfrak{m})}\zeta_k^*(1)X
\sum_{\substack{N(\mathfrak{d})\leq X^{1/2}\\(\mathfrak{d},\mathfrak{m})=1}}
\frac{\mu(\mathfrak{d})}{N(\mathfrak{d})^2}
\\
&\qquad=
\frac{\Phi(\mathfrak{m})}{N(\mathfrak{m})}\frac{\zeta_k^*(1)}{\zeta_k^\mathfrak{m}(2)}X
+O_n((\log |\Delta_k|)^{n-1} X^{1/2}(\log X)^{n-1})
\,;
\end{align*}
here we have used $\zeta^*_k(1)\ll (\log|\Delta_k|)^{n-1}$ (see, for example,~\cite{louboutin})
and
$$
\sum_{N(\mathfrak{d})>X^{1/2}}\frac{1}{N(\mathfrak{d})^{2}}\ll_n \frac{(\log X)^{n-1}}{X^{1/2}}
\,,
$$
which follows via partial summation from (\ref{E:BS}).

The second summand (\ref{E:another.second}) becomes
\begin{equation}\label{E:finale}
O_n\left(
\tau(\mathfrak{m})
(N(\mathfrak{f})|\Delta_k|)^{\frac{1}{n+1}}
X^{1-\frac{2}{n+1}}
(\log X)^{n-1}
\sum_{N(\mathfrak{d})\leq X^{1/2}}
\frac{1}{N(\mathfrak{d})^{2\left(1-\frac{2}{n+1}\right)}}
\right)
\,,
\end{equation}
and 
the last factor on the right above becomes $g_{\beta}(X^{1/2})$
for $\beta=-2+\frac{4}{n+1}$.
We now split into cases based on the value of $n$.
To complete the proof, it only remains to estimate $g_{\beta}(Z)$ appropriately in each case.

When $n\geq 4$, one has $\beta\leq -6/5$ and hence the sum converges and is bounded as
$g_\beta(Z)\leq\zeta_k(6/5)\leq\zeta(6/5)^n$.  This completes the proof in this case.

Before proceeding with the remaining cases,
we first claim that
\begin{equation}\label{E:claim}
g_0(Z)\ll_n(\log |\Delta_k|)^{n-1} Z\,,
\end{equation}
regardless
of the value of $n$.  Indeed, when $Z\geq|\Delta_k|$,
Equation (\ref{E:ltt3})
immediately gives the result upon application of
$\zeta_k^*(1)\ll (\log |\Delta_k|)^{n-1}$.
When $Z<|\Delta_k|$, we use the estimate (\ref{E:BS}).
This proves the claim.

Suppose $n=2$.
Using $g_0(Z)\ll Z\log|\Delta_k|$ and
applying
partial summation we obtain, for $-1<\beta<0$,
$g_\beta(Z)\ll Z^{\beta+1}\log |\Delta_k| $.
As $n=2$ leads to $\beta=-2/3$, this applies and we arrive at a similar 
conclusion as when $n>3$, but with the error term multiplied by a factor of
$X^{1/6}\log |\Delta_k|$.  This proves the result in the case where $n=2$.

Suppose $n=3$.
In this case, $\beta=-1$ and partial summation
yields $g_\beta(Z)\ll (\log |\Delta_k|)^2 \log Z$.  Hence we arrive at the same conclusion but with the error
term multiplied by $(\log |\Delta_k|)^2 \log X$, which proves the result in the final case.
\end{proof}

\section{Proof of Theorem~\ref{T:relative}}


We now proceed with the proof of Theorem~\ref{T:relative}.
We can see via Lemma~\ref{L:third} that
the contribution  from the principal character in (\ref{E:4}) from Lemma~\ref{L:first} is

\begin{align*}
&
2^{r_1(k)+r_2(k)}
\sum_{\substack{\mathfrak{d}\mid 2}}
\frac{1}{N(\mathfrak{d})}
\sum_{\substack{\mathfrak{c}\mid \mathfrak{d}}}
\mu\left(\frac{\mathfrak{d}}{\mathfrak{c}}\right)
\sum_{\substack{\mathfrak{a}\text{ squarefree}\\
N(\mathfrak{a})\leq \frac{N(\mathfrak{c}^2)}{4^n}X}}
\chi_0(\mathfrak{a})
\\[1ex]
&=
\frac{1}{2^{r_2(k)}}
\frac{\zeta_k^*(1)}{\zeta_k(2)}
X
+
\begin{cases}
O\left(|\Delta_k|^{1/3}\log|\Delta_k|X^{1/2}\log X\right)&n=2\,,\\
O\left(|\Delta_k|^{1/4}(\log|\Delta_k|)^2X^{1/2}(\log X)^3\right)&n=3\,,\\
O_n\left(|\Delta_k|^{\frac{1}{n+1}}X^{\errorexp}(\log X)^{n-1}\right)
&n>3\,,
\\
\end{cases}
\end{align*}
where the constant in the main term comes from
\begin{align*}
&
\frac{2^{r_1(k)+r_2(k)}}{4^n}
\left(
\sum_{\mathfrak{d}\mid 2}
\frac{\Phi(\mathfrak{d})}{N(\mathfrak{d})^2}
\prod_{\mathfrak{p}\mid\mathfrak{d}}
\left(1-N(\mathfrak{p})^{-2}\right)^{-1}
\sum_{\mathfrak{c}\mid\mathfrak{d}}
\mu\left(
\frac{\mathfrak{d}}{\mathfrak{c}}
\right)
N(\mathfrak{c})^2
\right)
\frac{\zeta_k^*(1)}{\zeta_k(2)}
\\
&
=
\frac{2^{r_1(k)+r_2(k)}}{4^n}
\left(
\sum_{\substack{\mathfrak{d}\mid 2}}
\Phi(\mathfrak{c})
\right)
\frac{\zeta_k^*(1)}{\zeta_k(2)}
\\
&=
\frac{1}{2^{r_2(k)}}
\frac{\zeta_k^*(1)}{\zeta_k(2)}
\,.
\end{align*}
Notice that in this case $\mathfrak{m}=\mathfrak{d}^2$ with $\mathfrak{d}\mid 2$
so that
$\tau(\mathfrak{m})\leq \tau(4)\ll_n 1$,
and consequently the dependence on
$\tau(\mathfrak{m})$ is absorbed into the implicit constant.

In our next estimate, it will be useful to note that the total number of characters being summed over in (\ref{E:4}) equals
$$
|\Cl_{\mathfrak{d}^2}(k)[2]|
\leq
|\Cl(k)[2]|\cdot\Phi(\mathfrak{d}^2)\leq |\Cl(k)[2]|\cdot N(\mathfrak{d})^2
\,.
$$

For simplicity, we first assume $n\geq 4$.  Again applying Lemma~\ref{L:third},
we find that the contribution in (\ref{E:4}) from the nonprincipal characters is bounded by
\begin{align}
\label{E:approp}
&\ll_n
2^{r_1(k)+r_2(k)}
\sum_{\mathfrak{d}\mid 2}
\frac{1}{N(\mathfrak{d})}
\sum_{\substack{\chi\in\widehat{\Cl_{\mathfrak{d}^2}(k)}\\\chi^2=\chi_0\\\chi\neq\chi_0 }}
\sum_{\substack{\mathfrak{c}\mid \mathfrak{d}}}
|\Delta_k|^{\frac{1}{n+1}}X^{\errorexp}(\log X)^{n-1}
\\
\label{E:approp2}
&\ll_n
|\Cl(k)[2]|
\cdot
|\Delta_k|^{\frac{1}{n+1}}X^{\errorexp}(\log X)^{n-1}
\,.
\end{align}
As before, some simplifications occur because we are allowed to drop dependence on $n$.
In particular, we have $N(\mathfrak{f})\leq N(\mathfrak{m})\leq 4^n$.
In the cases of $n=2$ and $n=3$, one simply modifies (\ref{E:approp}) and (\ref{E:approp2}) by
substituting the appropriate expression from the error term in Lemma~\ref{L:third}.

\section{Proof of Theorem~\ref{T:1}}\label{S:proof}
In this section, $k$ will always denote a quadratic field.
First note that Gauss' genus theory (see, for example~\cite{buell}) tells us that
\begin{equation}\label{E:Gauss}
|\Cl(k)[2]|\leq 2^{\omega(\Delta_k)-1}\ll |\Delta_k|^\eps
\,.
\end{equation}
In the case where $k$ is quadratic, Theorem~\ref{T:relative} gives
\begin{equation}\label{E:N1}
N_k(Y):=
\sum_{\substack{[K:k]=2\\N(\Delta_{K/k})\leq Y}}1=
\frac{1}{2^{r_2(k)}}
\frac{\zeta_k^*(1)}{\zeta_k(2)}Y
+
O(|\Delta_k|^{1/3+\eps}Y^{1/2}\log Y)
\,.
\end{equation}
We also have the weaker estimate
\begin{equation}\label{E:N2}
N_k(Y)\leq\sum_{\mathfrak{a}\in A}\sum_{\substack{\overline{u}\in S(k)\\N(\mathfrak{a})\leq Y}}1\leq |S(k)|
\sum_{\substack{N(\mathfrak{a})\leq Y}}
1
\ll
2^{\omega(\Delta_k)}Y\log |\Delta_k| 
\,,
\end{equation}
which one can derive from (\ref{E:0}), Lemma 3.2 of~\cite{cohen.diaz.olivier}, and (\ref{E:claim}).
We write
\begin{align}
\label{E:err.def}
&N_k(Y)=\frac{1}{2^{r_2(k)}}
\frac{\zeta_k^*(1)}{\zeta_k(2)}Y
+E_k(Y)\,,
\end{align}
and the left-hand side of (\ref{E:fields}) becomes
\begin{align*}
&
\sum_{\substack{[k:\Q]=2\\|\Delta_{k}|\leq \sqrt{X}}}
N_k(X/\Delta_k^2)
=
X
\sum_{\substack{[k:\Q]=2\\|\Delta_{k}|\leq \sqrt{X}}}
\frac{1}{2^{r_2}\Delta_k^2}
\frac{\zeta_k^*(1)}{\zeta_k(2)}
+
\sum_{\substack{[k:\Q]=2\\|\Delta_{k}|\leq \sqrt{X}}}
E_k(X/\Delta_k^2)
\,.
\end{align*}
For $1\leq Z\leq\sqrt{X}$, we have
\begin{align*}
\sum_{\substack{[k:\Q]=2\\|\Delta_{k}|\leq \sqrt{X}}}
E_k(X/\Delta_k^2)
&\ll
\sum_{\substack{[k:\Q]=2\\|\Delta_{k}|\leq Z}}
|\Delta_k|^{1/3+\eps}
\left(\frac{X}{\Delta_k^2}\right)^{1/2+\eps}
+
\sum_{\substack{[k:\Q]=2\\Z<|\Delta_{k}|\leq \sqrt{X}}}
|\Delta_k|^\eps\left(
\frac{X}{\Delta_k^2}\right) 
\\
&
\ll
X^{1/2+\eps}Z^{1/3-\eps}+X Z^{-1+\eps}
\,.
\end{align*}
Choosing $Z=X^{3/8}$, we obtain
\begin{equation}\label{E:done}
\sum_{\substack{[k:\Q]=2\\
|\Delta_{k}|\leq \sqrt{X}}}
\;\;
\sum_{\substack{[K:k]=2\\N(\Delta_{K/k})\leq X/\Delta_k^2}}1
=
X
\sum_{\substack{[k:\Q]=2}}
\frac{1}{2^{r_2}\Delta_k^2}
\frac{\zeta_k^*(1)}{\zeta_k(2)}
+O(X^{5/8+\eps})
\,.
\end{equation}
Here we have used
$$
\sum_{\substack{[k:\Q]=2\\|\Delta_{k}|> \sqrt{X}}}
\frac{1}{2^{r_2}\Delta_k^2}
\frac{\zeta_k^*(1)}{\zeta_k(2)}
\ll
\sum_{\substack{[k:\Q]=2\\|\Delta_{k}|> \sqrt{X}}}
\frac{\log|\Delta_k|}{\Delta_k^2}
\ll
X^{-1/2+\eps}
\,.
$$
Equations~(\ref{E:fields}) and~(\ref{E:done}) establish Theorem~\ref{T:1}.

\begin{remark}\label{R:1}
Note that an improvement of the error term in (\ref{E:done}) to $O(X^{1/2+\eps})$ would
prove the conjecture mentioned in~\S\ref{S:intro}.  Moreover, an improvement to $O(X^{1/2}\log X)$,
if possible, would instantly extract the secondary term as
$$
  N_4(D_4,X)=CX-\frac{3}{2}DX^{1/2}(\log X)^2+O(X^{1/2}\log X)
  \,,
$$
where the constant $D$ comes from the count $N_4(V_4, X)\sim DX^{1/2}(\log X)^2$.
\end{remark}

\section{Proof of Theorem~\ref{T:3}}
An easy variation on the discussion in \S\ref{S:initial}
allows us to see that in order to count $D_4$-extensions,
as before, it suffices to study extensions $K/k/F$
that are quadratic over quadratic.
For bookkeeping purposes, note that $[k:\Q]=2[F:\Q]=2n$.
First, we invoke
Theorem~2.1~of~\cite{MR4428868} to bound the $2$-torsion
of the class group of $k$ as
\begin{equation}\label{E:genus.new}
|\Cl(k)[2]|\ll_n |\Cl(F)[2]|^2|\Delta_k|^\eps\,.
\end{equation}

We adopt the notation from the proof of Theorem~\ref{T:1}.
The number of quadratic over quadratic extensions $K/k/F$ equals
\begin{align}
&
\label{E:once.again}
\sum_{\substack{[k:F]=2\\N_{F/\Q}(\Delta_{k/F})\leq\sqrt{X}}}
N_k\left(\frac{X}{N_{F/\Q}(\Delta_{k/F})^2}\right)
\\
\label{E:const.trunc}
&\qquad=
\sum_{\substack{[k:F]=2\\N_{F/\Q}(\Delta_{k/F})\leq\sqrt{X}}}
\frac{1}{2^{r_2(k)}}\frac{\zeta^*_k(1)}{\zeta_k(2)}\frac{X}{N_{F/\Q}(\Delta_{k/F})^2}
\\
\label{E:error.part}
&\qquad\qquad
+
\sum_{\substack{[k:F]=2\\N_{F/\Q}(\Delta_{k/F})\leq\sqrt{X}}}
E_k\left(\frac{X}{N_{F/\Q}(\Delta_{k/F})^2}\right)
\,.
\end{align}

Since $[k:\Q]=2n\geq 4$, Theorem~\ref{T:relative} combined with
(\ref{E:genus.new})
implies
\begin{equation}\label{E:analogue1}
  N_k(Y)=\frac{1}{2^{r_2(k)}}\frac{\zeta^*_k(1)}{\zeta_k(2)}Y+
  O_n(|\Cl(F)[2]|^2|\Delta_k|^{\frac{1}{2n+1}+\eps} Y^{1-\frac{2}{2n+1}+\eps})
  \,,
\end{equation}
which is the analogue of (\ref{E:N1}) in this context.
Additionally, we find
\begin{equation}\label{E:latter}
  N_F(Y)\ll_n |\Cl(F)[2]|\cdot |\Delta_F|^\eps Y 
  \,,
\end{equation}
which is the analogue of (\ref{E:N2}).
Using (\ref{E:analogue1}) we see that (\ref{E:error.part}) is
\begin{equation}\label{E:key}
  \ll_n
  |\Cl(F)[2]|^2
  |\Delta_F|^{\frac{2}{2n+1}+\eps}
  X^{1-\frac{2}{2n+1}+\eps}
  \hspace{-3ex}
\sum_{\substack{[k:F]=2\\N_{F/\Q}(\Delta_{k/F})\leq\sqrt{X}}}
\hspace{-3ex}
N_{F/\Q}(\Delta_{k/F})^{-2+\frac{5}{2n+1}-\eps}
\,.
\end{equation}
Here we have used the identity $|\Delta_k|=|\Delta_F|^2N_{F/\Q}(\Delta_{k/F})$.
Using partial summation on (\ref{E:latter}) one finds,
for fixed $\beta<-1$,
\begin{equation}\label{E:partial}
\sum_{\substack{[k:F]=2\\N_{F/\Q}(\Delta_{k/F})\leq Y}}
N_{F/Q}(\Delta_{k/F})^{\beta}
\ll_n
|\Cl(F)[2]|\cdot
|\Delta_F|^{\eps}
\,.
\end{equation}
Applying (\ref{E:partial}) to 
(\ref{E:key}) now gives the result.

\begin{remark}
If one prefers a version of Theorem~\ref{T:3} in which the dependence on $F$
is expressed purely in terms of the discriminant, one simply bounds
$|\Cl(F)[2]|$ in terms of $|\Delta_F|$.  Indeed,
$|\Cl(F)[2]|\ll_n |\Delta_F|^{\alpha(n)+\eps}$,
where the Brauer--Siegel Theorem allows one to take $\alpha(n)=1/2$.
More recently, it was shown (see~\cite{BSTTTZ}) that one can take
$\alpha(n)=1/2-1/(2n)$ when $n\geq 5$ and $\alpha(3)=\alpha(4)=0.2785$.
Of course, from (\ref{E:Gauss}) one can take $\alpha(2)=0$.
As an example, when $n=3$, one obtains an error of $O(|\Delta_F|^{1.13}X^{5/7+\eps})$.
\end{remark}

\vspace{2ex}
\noindent {\bf Acknowledgements.}
The authors would like to thank Frank Thorne for his helpful comments.
The authors are indebted to the anonymous referee for numerous suggestions
that greatly improved the quality of this paper.
The bulk of this work was completed
at the Mathematical Sciences Research Institute / Simons Laufer Mathematical Sciences Institute in Berkeley, CA during Spring 2023.  We would like to thank MSRI/SL Math for providing an excellent work environment and the National Science Foundation (under grant 1440140) for supporting MSRI.

\bibliographystyle{alpha}
\bibliography{counting2}


\pagebreak

\footnotesize{
\noindent
Kevin J. McGown\\
Department of Mathematics and Statistics\\
California State University, Chico\\
Chico, California 95929\\
U.S.A.\\
\emph{E-mail address:} {\tt kmcgown@csuchico.edu}\\[1ex]

\noindent
Amanda Tucker\\
Department of Mathematics\\
University of Rochester\\
Rochester, New York 14627\\
U.S.A.\\
\emph{E-mail address:} {\tt amanda.tucker@rochester.edu}
}

\end{document}